\documentclass[12pt,letterpaper]{article}

\usepackage{amsmath}
\usepackage{amsthm}
\usepackage{amssymb}
\usepackage{amsfonts}
\usepackage{url}
\usepackage{graphicx}

\newtheorem{prop}{Proposition}
\newtheorem{cor}[prop]{Corollary}
\newtheorem{theorem}[prop]{Theorem}
\newtheorem{lemma}[prop]{Lemma}

\DeclareMathOperator{\CM}{CM}

\begin{document}
\title{Cookie Monster Plays Games}
\author{Tanya Khovanova \\ MIT \and Joshua Xiong \\ Acton-Boxborough Regional High School}

\maketitle

\begin{abstract}
We research a combinatorial game based on the Cookie Monster problem called the Cookie Monster game that generalizes the games of Nim and Wythoff. We also propose several combinatorial games that are in between the Cookie Monster game and Nim. We discuss properties of P-positions of all of these games.

Each section consists of two parts. The first part is a story presented from the Cookie Monster's point of view; the second part is a more abstract discussion of the same ideas by the authors.
\end{abstract}

\section{Cookie Monster wants to play games}

Cookie Monster likes cookies. His mommy used his love for cookies to teach him to think and to play some mathematical games. She set up the following system with cookies. Cookie Monster's Mommy has a set of $k$ jars filled with cookies. In one move she allows Cookie Monster to choose any subset of jars and take the same number of cookies from each of those jars. Cookie Monster always wants to empty all of the jars in as few moves as possible. 

For example, if there are three jars with 1, 2, and 4 cookies he needs three moves. He can empty them one jar at a time. Or he can take one cookie from all of the jars in the first move, after that he will still need two more moves. But if the three jars have 1, 2, and 3 cookies, he can empty them in two moves. In the first move he can take one cookie from the first jar and the third jar. After that the two non-empty jars have 2 cookies in each. So he can empty the whole set of jars in one more move.

If there are $k$ jars with distinct number of cookies it is always possible to empty them in $k$ moves. Cookie Monster Mommy tries to make it interesting and sets up jars so that it is always possible to empty $k$ jars with distinct number of cookies in fewer than $k$ moves. For example, once she arranged the Fibonacci sequence of cookies in jars: $\{1,2,3,5,8\}$. Cookie Monster figured out how to empty the jars in 3 moves.

Cookie Monster Mommy tries to invent interesting sequences of numbers to use as the number of cookies in the jars and Cookie Monster tries to find the smallest number of moves. This is like a game.

But still, something was missing in this game. His mommy was in charge of the cookies, and he tried to solve her puzzles. Cookie Monster realized that he wants a game with his mommy, where he feels equal: a game in which two people have the same options at each move.

\subsection{Authors' comments}

Mathematicians now call the smallest number of moves for a given set $S$ of cookies in jars the Cookie Monster number of the set $S$. It is denoted as $\CM (S)$. The problem of finding the Cookie Monster number of a set of jars is called the Cookie Monster Problem. The problem first appeared in \textit{The Inquisitive Problem Solver} by Vaderlind, Guy, and Larson \cite{VGL}. Mathematicians got interested and wrote papers about the cookie monster problem and the cookie monster number \cite{B, BK, BrK}. Now eating cookies is not enough for the monster. The mathematical name for his new interest is an \textit{impartial combinatorial game}, a game in which two players each have the same moves available on any turn.

As we will see, Cookie Monster discovers the games of Marienbad and Nim in Sections~\ref{marienbad}~and~\ref{nim}. In Section~\ref{twojars}, Cookie Monster invents how to convert the Cookie Monster problem into a game. Cookie Monster tries the simplest case with two jars first, and then he finds out that the game is already known as Wythoff's game. In Section~\ref{cmgame}, Cookie Monster examines the Cookie Monster game with three jars, which is a previously unknown game. In Sections~\ref{sumofgames}, \ref{oddgame}, and \ref{othergames}, Cookie Monster invents many variations of the Cookie Monster game and calculates their P-positions. In Section~\ref{bigpicture} Cookie Monster discovers properties of P-positions of all the games and finds out that the maximum element in a P-position is bounded in terms of other elements.

\section{The Game of Marienbad}\label{marienbad}

Cookie Monster started bugging his mommy for a game. But mommy wanted to watch a movie. Then Cookie Monster's Mommy said, ``I heard there is a game in the movie. Let me watch the movie; I will remember the game and then we can play it.'' Cookie Monster agreed. He even tried to watch the movie himself. The movie was called ``Last Year at Marienbad.'' But unfortunately, the movie wasn't animated, and even worse, it was black and white. He became bored in two minutes. But his mother promised to call him each time the game was played.

The mysterious man in the movie introduced the game by saying, ``I know a game I always win.'' The second man replied, ``If you can't lose, it's no game.'' And the mysterious man said, ``I can lose, but I always win.'' It intrigued Cookie Monster.

You can use any objects to play the game. In the movie they sometimes played it with matches. That was so cool. Cookie Monster's Mommy didn't allow him to touch matches. In the movie the men didn't light the matches, but still it was cool. There were four rows of matches, and there were 1, 3, 5, or 7 matches in each row as in Figure~\ref{fig:marienbad}. In one move, a player can choose a row and take any number of matches from that row. The player who is forced to take the last match loses. The mysterious man in the movie always won. 

\begin{figure}[htbp]
  \centering
    \includegraphics[scale=0.4]{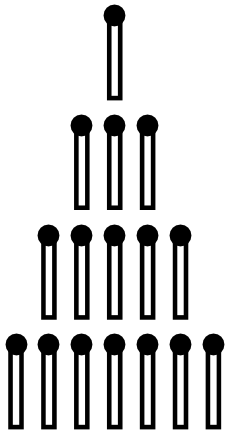}
      \caption{Initial setup in the game of Marienbad}\label{fig:marienbad}
\end{figure}

After the movie Cookie Monster disappeared, and his Mommy didn't go looking for him; she had so many other more important things to do than play games. Meanwhile, Cookie Monster tried to figure out the game. He realized that at the end of the game, when only piles of one match remained, he would win if he left an odd number of such piles. This means that if all but one of the piles have one match, he has a winning strategy: he can take either all of the matches from the largest pile, or all but one match from it to make sure that an odd number of piles of one match are left. This is his end-game.

Cookie Monster started analyzing the game from the end. He denotes a position using a list of numbers in parentheses to represent the number of matches in each pile. The starting position is (1,3,5,7). He figured out that if after his move the numbers in piles are doubled, as in (3,3) or (1,1,5,5), he can win. Whatever his opponent does, Cookie Monster can keep this double property, and eventually, Cookie Monster will be able to use his end-game strategy. Cookie Monster decided to call the positions that he needs to move to in order to win the game P-positions, from the word \textbf{p}ositive. He feels good and positive moving to a P-position. In addition, Cookie Monster called the other positions N-positions from the word \textbf{n}egative, since he does not want to move into these positions.

Cookie Monster tried to play with himself and found another P-position: (1,2,3). If he finishes in this position, he is guaranteed to win. He continued his search and found two more P-positions: (1,4,5) and (2,4,6). Curiously, the latter is a double of the P-position (1,2,3).

He found more P-positions, but he kept forgetting the exact numbers. The only thing that he remembers is if the first player starts removing one match from (1,3,5,7), he needs to remove one match from any other pile.

Now he is ready to play with his mommy.

\subsection{Authors' comments}

The game in the movie is known as the game of Marienbad. It is a variation of another game called Nim. In the game of Nim you can have any initial position. In addition, the winning condition for Nim is different: the person who makes the last move wins. The variation in the movie, where the person who takes the last match loses, is called mis\`{e}re.

The notion of P-positions and N-positions exists in game theory. Cookie Monster reinvented these names with the same meaning. But in game theory, the name ``P-position'' actually comes from the fact that the \textbf{p}revious player wins after leaving such a position, assuming the use of an optimal strategy. Similarly, the name ``N-position'' is due to the fact that the \textbf{n}ext player wins with the correct strategy. 

Not counting the very end of the game---when only piles of size one remain---the P-positions are the same for the regular Nim and the mis\`{e}re variation. Thus, the strategy is nearly the same for both versions; the only difference is in the end-game: In the standard game of Nim, a player needs to leave an even number of piles of size one (and no other piles) in order to win. In the mis\`{e}re variation, a player needs to leave an odd number of piles of size one in order to win.

To denote a position in a game, Cookie Monster uses an ordered tuple of numbers in parentheses. The starting position in Marienbad is (1,3,5,7), which is a P-position. The P-positions that Cookie Monster has trouble remembering are: (2,5,7), (3,4,7), (3,5,6), (1,2,4,7), (1,2,5,6), and (1,3,4,6).

There is actually a formula for P-positions of Nim; we will show it in the next section.

\section{Cookie Monster Plays Nim}\label{nim}

Now Cookie Monster is ready to play with his mommy. And mommy agreed to play with matches. Hooray! They started playing the game as it was set up in the movie and mommy was first to move. After Cookie Monster's Mommy lost several games she asked to start second. Cookie Monster couldn't delay this moment forever, so he agreed. That was a challenge. But Cookie Monster decided to start with a small move: the more matches are on the table the more difficult it is for his mom to figure out the winning strategy. Also, the longer the game goes on, the more moves both of them make, and the more chances there are for mommy to make a mistake.

Cookie Monster won again, and wanted to be the second player. His mommy realized that something fishy was going on: that the second player has an advantage. She refused to be the first player. Cookie Monster didn't want to risk losing, so he offered to play the standard game of Nim instead: the person who does not have a move loses. It seems the winning condition is the opposite, so one might think that the advantage moves from the second player to the first player. 

Mommy agreed to be the first player, and she lost three times in a row. She wanted to be the second player again. To divert her, Cookie Monster suggested a ``fair'' game: they will both choose something. His mommy chooses a starting position, then Cookie Monster chooses who goes first. Of course, Cookie Monster knew that this game was not fair. This is because Cookie Monster will choose to be second if his mommy chooses a P-position, and first otherwise. This way he would always win.

Cookie Monster's Mommy is very smart. She lost every game because she was busy baking cookies and did not have enough time to think about the game. This time, however, she saw right through him and didn't agree to his ``fair'' suggestion. Oh well. At least he can try this game on his friends.

\subsection{Authors' comments}

There is a formula for P-positions in the game of Nim. From now on we will talk about the standard game, where the player who cannot move loses. This means that the position with all piles empty is a P-position. 

In addition, we want to emphasize that any move from every P-position must end up in an N-position. And there should exist at least one move from every N-position to a P-position.

To state the formula, first we need to define the \textit{nim-sum}, or \textit{bitwise XOR} operation: $\oplus$. Nim-sum can be described as a binary addition without carry. In other words, the nim-sum of two numbers is produced by representing each number as a sum of distinct powers of two, canceling each power of two that appears twice, and adding the remaining powers of two. The nim-sum is a commutative and associative operation. 

\begin{theorem}[Bouton]
The P-positions in the game of Nim are formed by a set of numbers with nim-sum zero.
\end{theorem}

The proof can be found in many places \cite{BCG, Bouton}. We just want to mention that Cookie Monster is very perceptive. He noticed that P-position (2,4,6) is the doubling of each pile of the P-position (1,2,3). By the nim-sum property this is always true: doubling each element in a P-position will result in a P-position. 

Another observation of Cookie Monster was that if someone removes one match from one of the piles of the starting position (1,3,5,7), then to get to the next P-position he needs to remove one match from any other pile. This strategy works for any starting P-position where all piles are odd. 

Note the following corollary:

\begin{cor} 
Given the number of matches in all but one of the piles, there is a unique value of the number of matches in that pile that makes the position a P-position.
\end{cor} 

\begin{proof}
The last pile is the nim-sum of the other piles.
\end{proof}

\section{Cookie Monster Game with 2 Jars}\label{twojars}

Cookie Monster played the game of Nim with his friends and always won. But then he said to himself, ``There is a Cookie Monster problem, there should be a Cookie Monster game.''

Here is how the Cookie Monster problem is converted to a two-player game. There are several jars filled with cookies. In one move a player chooses any subset of jars and takes the same number of cookies from each of those jars. The player who cannot move loses.

Cookie Monster decided to study his own game. If there is one jar and he starts he can win with one move by taking all the cookies. What happens if there are two jars? 

If it were Nim with two jars, then the same number of cookies in both jars $(n,n)$ would constitute a P-position. But in the Cookie Monster game you can empty those two jars in one move, so this must be an N-position.

Cookie Monster calculated some small P-positions: $(1,2)$, $(3,5)$, $(4,7)$, $(6,10)$ and $(8,13)$. Some of them are pairs of consecutive Fibonacci numbers. Then he remembered where he saw other pairs. This is how he converts miles to kilometers.

Suppose you want to convert miles to kilometers. Take the number of miles, represent it as a sum of different Fibonacci numbers, then replace each Fibonacci number with the next one and sum the numbers to get your conversion. For example, 6---as a sum of different Fibonacci numbers---is $5+1$. Replacing each number by the next Fibonacci number, he finds that the new sum is $8+2=10$. The reverse conversion is similar: you just need to take the previous Fibonacci number instead of the next one.

\subsection{Authors' comments}

The Cookie Monster game with two jars was invented in 1907 and is called Wythoff's Game. The Cookie Monster game for more than two jars first appeared in literature more than 100 years later in 2013 \cite{B}.

Cookie Monster's method of converting miles to kilometers involves the Zeckendorf representation of a positive integer. To find the \textit{Zeckendorf representation}, take an integer $n$, subtract the largest Fibonacci number not greater than $n$, and repeat. As a result, $n$ is represented as a sum of $j$ distinct Fibonacci numbers: $n= F_{i_1} + F_{i_2} + \ldots + F_{i_j}$, where $i_1 < i_2 < \ldots < i_j$. If one of the Fibonacci numbers used is $1$, its index is defined to be $2$. In fact, there are no two neighboring Fibonacci numbers in this representation \cite{Z}. By construction, this representation is unique. Now we define the \textit{Fibonacci successor} of $n$, $\sigma(n)$, by shifting the index of Fibonacci numbers in the Zeckendorf representation \cite{CF}: 

$$
\sigma(n) = F_{i_1+1} + F_{i_2+1} + \ldots + F_{i_j+1}.
$$

The trick of converting miles to kilometers is based on the fact that the ratio of kilometers to miles is 1.609, which is very close to the golden ratio, 1.618. The ratio $F_{n+1}/F_n$ is very close to the golden ratio for $n \geq 5$. Thus for $n > 10$, the Fibonacci successor $\sigma(n)$ is  very close to $1.6n$. This justifies the miles to kilometers conversion. 

In order to reverse the conversion, converting kilometers to miles, we can use the same principle, but in reverse. However, this does not quite work because not every number is a successor of another number. Only numbers that do not have $1$ in their Zeckendorf representations are successors. But we still can convert kilometers to miles. If the number of kilometers has 1 in its Zeckendorf representation we can either keep it or replace by 0. In an approximate calculation, it does not matter.

The following well-known theorem describes P-positions in the Wythoff's Game (see \cite{Co, Wy}).

\begin{theorem}
All P-positions can be described in the form $(n,\sigma(n))$. Every positive integer appears in one P-position. If the index $i_1$ of the smallest Fibonacci number in the Zeckendorf representation of $n$ is even, then $n$ is the first number, otherwise it is the second number.
\end{theorem}

Cookie Monster was correct in observing that some P-positions involved consecutive Fibonacci numbers. The corollary describes exactly which Fibonacci pairs are P-positions.

\begin{cor}
$(F_{2n}, F_{2n+1})$ is a P-position, and no other P-positions involve Fibonacci numbers.
\end{cor}

\section{Cookie Monster Game}\label{cmgame}

Cookie Monster started studying his own game with 3 jars. He wrote a program and found some P-positions. The P-positions with one empty jar are the same as in the Wythoff game. Here some more P-positions with all the jars non-empty and each jar has less than 10 cookies: $(1, 1, 4)$, $(1, 3, 3)$, $(1, 5, 6)$, $(2, 2, 6)$, $(2, 3, 8)$, $(2, 7, 7)$, $(3, 4, 4)$, $(3, 6, 9)$, $(5, 5, 7)$, $(5, 8, 8)$. As jars are permutable, we only need to write one of the permutations.

It was difficult to calculate these positions, and Cookie Monster went online to try to find some literature on the subject, and found only one paper. M.~Belzner~\cite{B} already studied the Cookie Monster game with 3 jars. She tried to calculate P-positions for the case when one of three jars contains 1 cookie and the other two are not empty. But the problem is so difficult that she made a mistake. She listed $(1,7,9)$ as a P-position. Cookie Monster can prove that $(1,7,9)$ is not a P-position. He can take 7 cookies from the last two jars (and eat all 14 of them) to get to $(0,1,2)$, which is a  P-position. Consequently, all the following P-positions in~\cite{B} are wrong. The author probably was solving this manually. It is good that Cookie Monster can program.

\subsection{Authors' comments}

To our knowledge, \cite{B} is the only paper studying P-positions of the Cookie Monster game with more than 2 jars. That makes Cookie Monster the first person to correctly calculate the P-positions in this game.

\section{Nim and the First Two Jars}\label{sumofgames}

Cookie Monster realized how difficult the problem is and decided to invent other games. In Nim, each move involved one pile. In the Cookie Monster game, each move involves any subset of the piles. So Cookie Monster decided to do something in between. Each game will have \emph{permissible sets of jars} that the players can take from. In Nim, only sets of size 1 are permissible. In the Cookie Monster game, every subset of jars is permissible. His new games are between Nim and Cookie Monster. So all sets of size 1 are permissible, but not all subsets are permissible. The games can be generated by adding permissible sets to Nim or subtracting permissible sets from the Cookie Monster game.

Cookie Monster realized that there were many different games that he could make, so he invented a common name for them. He decided to call these new games \emph{Cookie-Monster-Nim games}, or \emph{CM-Nim games} for short.

Cookie Monster first decided to add exactly one set of jars from which he allowed to take the same number of cookies, since he thought adding fewer moves would make it easier to find P-positions.

First he tries to play the following game: You are allowed to take any number of cookies from individual jars. You are also allowed take the same number of cookies from both the first and the second jar. The P-positions with one empty jar will either be the same as Nim or as Wythoff P-positions, depending on which jar is empty. So he wrote a program to calculate some more P-positions where none of the jars is empty. He calculated the P-positions were none of the jars has more than six cookies: as the first and the second jars are interchangeable, we only need to write the positions such that the number of cookies in the first jar is not greater than the number in the second jar. These are the P-positions: $(1, 1, 2)$, $(1, 3, 4)$, $(1, 4, 5)$, $(1, 5, 3)$, $(2, 2, 1)$, $(2, 3, 5)$, $(2, 4, 3)$, $(2, 5, 4)$, $(3, 3, 6)$, $(3, 4, 2)$, $(3, 6, 1)$, $(4, 5, 6)$, $(6, 6, 3)$. 

Once again Cookie Monster went online and discovered that this game is actually known: it is the sum of two games: Wythoff's game plus Nim with one jar. So Cookie Monster lost interest in it.

\subsection{Authors' Comments}

The \textit{sum} of two given games is defined as follows: Two players are playing two games against one another. On each move a player decides which game to play and makes one move.

To find the P-positions of this game, we use the following theorem that was discovered about all impartial combinatorial games.

\begin{theorem}[Sprague-Grundy Theorem]
Every position of any impartial game is equivalent to a Nim heap.
\end{theorem}

The proof of this theorem \cite{Grundy, Sprague} is fairly technical, so we will omit it here. Accordingly, there is a Sprague-Grundy function, which takes a position of an impartial game and gives the equivalent Nim heap, which is called a nimber.

The nimber equivalent to the sum of several games is the nim-sum of the nimbers of the positions in each game. In order for this sum of games to be a P-position, the nim-sum of the nimbers of each individual game must be equal to zero. 

Cookie Monster cannot take from both the last jar and at least one of the first two in his new game. So this game is equivalent to the sum of Wythoff's game---the first two jars---and Nim with one jar---the third jar. 

The Sprague-Grundy function of Nim with one jar is, of course, the nimber of the same size. The Sprague-Grundy function of the Wythoff game is more complicated, and no explicit formula is known, although many properties have been deduced \cite{BF}. In P-positions in Cookie Monster's new game, the number of cookies in the third jar is exactly the Sprague-Grundy function of Wythoff's game with the first two jars.

\section{Nim and All Three Jars}\label{oddgame}

Cookie Monster tried a different game: the one where as an additional move you are allowed to take the same number of cookies from all the three jars. He calls this the \emph{odd CM-Nim game} as you always take from an odd number of jars. He discovered that the P-positions are exactly the same as in the game of Nim. How could this be? If we add more moves should the number of P-positions decrease? Actually there are infinitely many P-positions, but still he feels that the number of P-positions should decrease in some sense.

\subsection{Authors' Comments}

The notion of an odd CM-Nim game can be extended to any number of jars. In one move, a player is allowed to take the same number of cookies from an odd number of jars. Cookie Monster was correct in thinking that P-positions of an odd game are the same as P-positions of Nim.

\begin{theorem}
The odd CM-Nim game with $k$ jars has the same P-positions as Nim with $k$ piles.
\end{theorem}

\begin{proof}
If the nim-sum of all the piles is zero, then one move cannot preserve this property. That is any move from a P-position of Nim goes to an N-position of Nim. At the same time, in Nim there exists a move from an N-position to a P-position. The same move exists in the odd Cookie Monster game. That means the P-positions of the odd Cookie Monster game are the same as P-positions of Nim.
\end{proof}

Using similar reasoning, we can prove the following theorem:

\begin{theorem}
If you add moves to a game that only move from the old game's P-positions to the old game's N-positions, the new game has the same set of P-positions as the old game.
\end{theorem}

\section{Other Games with 3 Jars}\label{othergames}

What other moves could Cookie Monster add to the game of Nim or subtract from the Cookie Monster game to invent other new games? Not to miss anything, Cookie Monster first looks at the games when you can take from at most two jars. In addition to taking from one jar at a time, you are allowed to take from two jars. It produces two more games. For each game there is a list of permissible sets. To differentiate permissible sets of jars from P-positions, Cookie Monster uses curly brackets for the sets of jars.

\begin{itemize}
\item \textbf{At-Most-2-Jars}: From any two jars: $\{1,2\}$, $\{1,3\}$, and $\{2,3\}$. 
\item \textbf{Consecutive-At-Most-2-Jars}: $\{1,2\}$ and $\{2,3\}$.
\end{itemize}

What if the jars are not consecutive? Cookie Monster can relabel the jars any way he wants, so the game with permissible sets $\{1,2\}$ and $\{1,3\}$ is the same as Consecutive-At-Most-2-Jars.

Now Cookie Monster looks at the games where you are allowed to take from all three jars:
\begin{itemize}
\item \textbf{Consecutive-Include-First-Jar}: $\{1,2\}$, and $\{1,2,3\}$.
\item \textbf{Consecutive}: $\{1,2\}$, $\{2,3\}$ and $\{1,2,3\}$.
\end{itemize}

It is easy to calculate P-positions when one of the jars is empty because there are two jars left, and the game reduces to either Nim or Wythoff. Cookie Monster decides to compare the games by finding P-positions with 2 ones. He arranged them in Table~\ref{tbl:twoones}.

\begin{table}[htbp]
\begin{center}
  \begin{tabular}[htbp]{| l | r |}
    \hline
    Nim & (0,1,1) (1,0,1) (1,1,0) \\ \hline
    Wythoff plus Nim & (0,1,1) (1,0,1) (1,1,2)\\ \hline
    Odd Cookie Monster & (0,1,1) (1,0,1) (1,1,0) \\ \hline
    Cookie Monster & (4,1,1) (1,4,1) (1,1,4) \\ \hline
    At-Most-2-Jars & (1,1,1) \\ \hline
    Consecutive-At-Most-2-Jars & (3,1,1) (1,0,1) (1,1,3) \\ \hline
    Consecutive-Include-First-Jar & (0,1,1) (1,0,1) (1,1,2) \\ \hline
    Consecutive & (3,1,1) (1,0,1) (1,1,3) \\ 
    \hline
  \end{tabular}
\end{center}
\caption{P-positions in CM-Nim games with 2 ones}\label{tbl:twoones}
\end{table}

Cookie Monster is a proud programmer. He calculated more P-positions in hopes of finding patterns. These are the P-positions that he found where each jar is non-empty and has at most 6 cookies:

\begin{itemize}
\item At-Most-2-Jars.  P-positions are sorted as the order does not matter;: $(1, 1, 1)$, $(1, 3, 4)$, $(2, 2, 2)$, $(2, 3, 6)$, $(2, 5, 7)$, $(3, 3, 3)$, $(4, 4, 4)$, $(5, 5, 5)$, $(6, 6, 6)$.
\item Consecutive-At-Most-2-Jars. We can switch the first and the third jars in the P-positions, so the list has only P-positions, where the first jar has no more cookies than the last jar: $(1, 1, 3)$, $(1, 3, 2)$, $(1, 4, 4)$, $(2, 2, 5)$, $(2, 6, 3)$, $(2, 5, 6)$, $(3, 2, 4)$, $(3, 4, 5)$, $(4, 1, 6)$, $(4, 6, 5)$.
\item Consecutive-Include-First-Jar: Order matters in this case. $(1, 1, 2)$, $(1, 3, 4)$, $(1, 4, 5)$, $(1, 5, 3)$, $(2, 2, 1)$, $(2, 3, 5)$, $(2, 4, 3)$, $(2, 5, 4)$, $(3, 1, 4)$, $(3, 2, 5)$, $(3, 3, 6)$, $(3, 6, 1)$, $(4, 1, 5)$, $(4, 2, 3)$, $(4, 5, 2)$, $(5, 1, 3)$, $(5, 2, 4)$, $(5, 4, 2)$, $(6, 3, 1)$, $(6, 6, 3)$.
\item Consecutive. In this case, the first and the third jars are interchangeable, so the list has only P-positions where the first jar does not have more cookies than the last jar: $(1, 1, 3)$, $(1, 3, 6)$, $(1, 4, 4)$, $(1, 5, 2)$, $(1, 6, 5)$, $(2, 2, 5)$, $(2, 3, 3)$, $(2, 6, 4)$, $(3, 2, 4)$, $(4, 1, 6)$, $(5, 5, 6)$.
\end{itemize}

After staring at the P-positions, Cookie Monster notices only one pattern: In the game of At-Most-2-Jars with 3 jars, $(n,n,n)$ is always a P-position.

\subsection{Authors' comments}

Cookie Monster's observations about the At-Most-2-Jars game are correct. In fact, this game is a specific case of another game: the game with $k$ jars called \textit{All-but-$k$}, where a player can take any number of cookies from any subset of jars except from all of the $k$ jars.

\begin{theorem}
In the All-but-$k$ game, $(n,n,\ldots,n)$ is a P-position for any $n$. Further, if a position has $|\{a_1, \ldots, a_k\}| = 2$, it is an N-position.
\end{theorem}

\begin{proof}
By induction. Let $\overline{\mathcal{P}}$ be the set of our candidate P-positions, and let $\overline{\mathcal{N}}$ be the set of our candidate N-positions. Now we show that three properties hold:
\begin{enumerate}
\item from each P-position in $\overline{\mathcal{P}}$, we can only move to N-positions in $\overline{\mathcal{N}}$, 
\item from every N-position in $\overline{\mathcal{N}}$, there exists a move to a P-position in $\overline{\mathcal{P}}$,
\item the terminal position is in $\overline{\mathcal{P}}$.
\end{enumerate}

If all the jars have the same number $n$ of cookies, and we take $c$ from some of the jars, then after the move the number of cookies is either $n$ or $n-c$, and both numbers exist. This proves the first statement. 

If some of the jars have $n$ cookies and other jars have $n-c$ cookies. By taking $c$ cookies from all of the jars with $n$ cookies, all of the jars will have $n-c$ cookies, proving the second statement.

The third statement is trivial.

So if we start with any position in $\overline{\mathcal{P}}$, our opponent can only move to a position in $\overline{\mathcal{N}}$. We can then move back to a position in $\overline{\mathcal{P}}$ with a smaller total number of cookies. Eventually we will move to the terminal P-position and win.
\end{proof}

We can generalize it even further:

\begin{theorem}
In a game where a complement of every permissible set is permissible, $(n,n,\ldots,n)$ is a P-position for any $n$.
\end{theorem}

\section{The Big Picture}\label{bigpicture}

Cookie Monster is very pleased that he invented many new games. He is also proud that he discovered some properties of some of CM-Nim games. But mathematicians also like to look at the big picture: are there any properties that hold for all of the CM-Nim games, and not just for specific games?

Cookie Monster noticed that in all the games the last value of a P-position is uniquely defined by the previous terms. He also noticed that one of the numbers is never more than twice bigger than the other two. This is his contribution to the big picture.

\subsection{Authors' comments}
Cookie Monster's observation about the dependency of the last jar on the other jars is correct. The following theorem is true for all CM-Nim games with $k$ jars.

\begin{theorem}
In any CM-Nim game, for a position with the number of cookies in $k-1$ jars known and one unknown: $P = (a_1, \ldots, a_{k-1}, x)$, there is a unique value of $x$ such that the $P$ is a P-position.
\end{theorem}

\begin{proof}
The proof consists of two parts.

\emph{Part 1: Two P-positions cannot differ in the last number only}. 

We can always make a move from a position with a greater value of $x$ to the position with the smaller value of $x$, contradicting the fact that a P-position can only move to N-positions.

\emph{Part 2: There exists $x$ such that $P$ is a P-position}. 

Suppose that all of the positions $(a_1, \ldots, a_{k-1}, x)$ are N-positions, which means from each of them, there exists a move to a P-position. Clearly, this move cannot involve only taking from the last jar by our assumption, so our move must involve taking from at least one of the first $k-1$ jars. 

Now we will give an upper bound $U$ on the number of possible moves that there are from these positions that could lead to P-positions. The maximum number of permissible sets is $2^k - 1$. Since we cannot take from the last jar only, we can take at most $m = \max\{a_1, \ldots, a_{k-1}\}$ cookies from each jar, so $U \leq m(2^k - 1)$.

Consider the following assumed N-positions: $(a_1, \ldots, a_{k-1}, x)$, where $0 \leq x \leq U$. Applying the Pigeonhole principle, we see there must be two positions where we used the same move to get to a P-position. Suppose that this move removes $b_i$ cookies from jar $i$, and let the two positions be $(a_1, \ldots,  a_{k-1}, A)$ and $(a_1, \ldots,  a_{k-1}, B)$. Then the P-positions where we moved to are $(a_1 - b_1, \ldots,  a_{k-1} - b_{k-1}, A - b_k)$ and $(a_1 - b_1, \ldots,  a_{k-1} - b_{k-1}, B - b_k)$. However, by part 1, this is a contradiction, which means that for some value of $x$, $(a_1, \ldots, a_{k-1}, x)$ is a P-position.
\end{proof}

We can easily generalize the proof for any combination of $n-1$ jars, not necessarily the first $n-1$. In other words, the $i$th number in a P-position is a function of the other numbers.

Cookie Monster is also correct about the bound on the largest number in P-positions of a CM-Nim game. To prove this bound, we need the following lemma:

\begin{lemma}\label{samejar}
Suppose that we start with a P-position, and a player takes from a set $S$ of jars. Then the other player cannot take from the exact same set $S$ of jars to get to a P-position.
\end{lemma}
\begin{proof}
Suppose that the other player can do so. But this means that the first player could have moved to this P-position instead, which is a contradiction.
\end{proof}

Now we are able to prove the bound.

\begin{theorem}
In a P-position of a CM-Nim game with $k$ jars, one of the numbers is never larger than twice the sum of the other numbers.
\end{theorem}

\begin{proof}
Suppose that a P-position is of the form $(a_1, a_2, \ldots, a_k)$. It suffices to prove the statement for the largest value of $a_i$, where $0 \leq i \leq k$, which we will call $a_j$. We can write this inequality as $2(\sum_{i=1}^n a_i - a_j) \geq a_j$.

We prove this statement by strong induction on the sum $T = \sum_{i=1}^{k}  a_i - a_j$. Our base case is $T=0$, which is obvious. Now suppose that the statement is true for $T \leq n$. We want to prove the statement for $T = n+1$.

We proceed by contradiction. Suppose the position $P=(a_1, a_2,\ldots, a_k)$ is a P-position, and $2(n+1) < a_j$. All possible moves from $P$ must be to N-positions. Thus, the move that takes exactly one cookie from jar $j$ will result in an N-position.

From this N-position, we can then move to a P-position. If we want to get to a P-position, we cannot only take cookies from jar $j$ by Lemma~\ref{samejar}. Suppose that the P-position that we move to is $P' = (a'_1, a'_2, \ldots, a'_k)$, and let $T' = \sum_{i=1}^{k}  a'_i - a'_j$. If we do not take from jar $j$ in our second move, we will have $2T' \le 2n < 2n + 2 - 1 < a_j - 1$, which, by the inductive hypothesis is a contradiction. Now suppose that we take $c$ cookies from jar $j$ as well as from $d$ of the other jars. Then $2cd \geq 1+c$, with equality if $c=d=1$, so $2T' = 2(n+1 - cd) < a_j - 1 - c$, which is a contradiction, and we are done.
\end{proof}

\section{Acknowledgements}\label{acknowledgements}

Cookie Monster and the authors are grateful to the MIT-PRIMES program for supplying cookies and supporting this research.


\begin{thebibliography}{99}

\bibitem{B} M.~Belzner, \textit{Emptying Sets: The Cookie Monster Problem}, {\tt arXiv:1304.7508 [math.CO]}, 2013.

\bibitem{BCG} E.~R.~Berlekamp, J.~H.~Conway and R.~K.~Guy, \emph{Winning Ways for your Mathematical Plays}, vol. 1, second edition, A.~K.~Peters, 2001.

\bibitem{BK} O.~Bernardi and T.~Khovanova, The Cookie Monster Problem, available at \url{http://blog.tanyakhovanova.com/?p=325}, 2011.

\bibitem{Bouton} C.~L.~Bouton. Nim, a game with a complete mathematical theory. \emph{Ann. of Math.} \textbf{3} (1901--1902), 35--39.

\bibitem{BrK} L.~M.~Braswell and T.~Khovanova, Cookie Monster Devours Naccis, \textit{Coll. Math. J.} \textbf{45} (2014), 129--135.

\bibitem{BF} U.~Blass and A.~S.~Fraenkel, The Sprague-Grundy function for Wythoff's game, \textit{Theoret. Comput. Sci.} \textbf{75} (1990), 311--333.

\bibitem{CF} P.~J.~Cameron and G.~Fon-Der-Flaass, Fibonacci Notes, available at \url{http://www.maths.qmul.ac.uk/~pjc/comb/fibo.pdf}.

\bibitem{Co} H.~S.~M.~Coxeter, The golden section, phyllotaxis, and Wythoff's game, \textit{Scripta Mathematica} \textbf{19} (1953) 135--143.

\bibitem{Grundy} P.~M.~Grundy, Mathematics and games, \textit{Eureka} \textbf{2} (1939), 6--8; reprinted \textbf{27} (1964), 9--11.

\bibitem{Sprague} R.~P.~Sprague, \"Uber mathematische Kampfspiele, \textit{T\^ohoku Math. Journal} \textbf{41} (1935--1936) 438--444.

\bibitem{VGL} P.~Vaderlind, R.~K.~Guy, and L.~C.~Larson. \textit{The Inquisitive Problem Solver}. The Mathematical Association of America, Washington, D.C., 2002.

\bibitem{Wy} W.~A.~Wythoff, A Modification of the Game of Nim, \textit{Nieuw Arch. Wisk.} \textbf{7} (1907), 199--202. 

\bibitem{Z} E.~Zeckendorf. Repr\'{e}sentation des nombres naturels par une somme des nombres de Fibonacci ou de nombres de Lucas, \textit{Bull. Soc. Roy. Sci. Li\`{e}ge} \textbf{41} (1972), 179--182.

\end{thebibliography}
\end{document}